\documentclass[12pt,centertags,oneside]{amsart}
\usepackage{stmaryrd}
\usepackage{amsmath,amstext,amsthm,amscd,typearea,hyperref}
\usepackage{amssymb}
\usepackage{esint}
\usepackage{a4wide}
\usepackage[mathscr]{eucal}
\usepackage{mathrsfs}
\usepackage{typearea}
\usepackage{charter}
\usepackage{pdfsync}
\usepackage[a4paper,width=16.5cm,top=3cm,bottom=3cm]{geometry}

\numberwithin{equation}{section}

\usepackage[english]{babel}

\allowdisplaybreaks
\emergencystretch=\maxdimen
\usepackage{multicol}

\usepackage{xcolor}

\newtheorem{theorem}{Theorem}[section]

\newtheorem{proposition}[theorem]{Proposition}
\newtheorem{corollary}[theorem]{Corollary}

\newtheorem{remark}[theorem]{Remark}

\newtheorem{example}[theorem]{Example}

\newtheorem*{definition*}{Definition}

\newcommand{\cali}[1]{\mathscr{#1}}

\newcommand{\Per}{{\rm Per}}
\newcommand{\Tan}{{\rm Tan}}

\newcommand{\Cc}{\cali{C}}

\newcommand{\FS}{{\rm FS}}

\newcommand{\C}{\mathbb{C}}

\newcommand{\R}{\mathbb{R}}

\renewcommand\P{\mathbb{P}}

\newcommand{\Gal}{\operatorname{Gal}}

\newcommand{\vphi}{\varphi}

\newcommand{\ovl}{\overline}

\newcommand{\cal}{\mathcal}

\newcommand{\eps}{\varepsilon}

\newcommand{\h}{\widehat{h}}

\newcommand{\bb}{\mathbb}

\usepackage[
backend=biber,
style=alphabetic,
]{biblatex}

\usepackage{fancyhdr}

\title{Lower Bounds for Galois Orbits of periodic points for polarized endomorphisms}

\author{Tien-Cuong Dinh}
\address{Department of Mathematics,  National University of Singapore - 10, Lower Kent Ridge Road - Singapore 119076}
\email{matdtc@nus.edu.sg}

\author{Jit Wu Yap}
\address{182 Memorial Drive, Cambridge, MA 02139}
\email{jitwuyap@mit.edu}

\begin{document}

\begin{abstract}
Let $K$ be a number field, $X$ a smooth projective variety over $K$ and $f: X \to X$ a polarized endomorphism of degree $d \geq 2$. We prove an exponential lower bound on $[K(\Per_n):K]$, where $\Per_n$ is the set of $n$-periodic points, extending results of \cite{Yap24} to higher dimensions. We also prove a quantitative rate of equidistribution for $\Per_n$ to the equilibrium measure.  
\end{abstract}
   
\clearpage\maketitle

\section{Introduction}
Let $K$ be a number field and $X$ a smooth projective variety over $K$. Let $f: X \to X$ be a polarized endomorphism of degree $d \geq 2$. The aim of this short paper is to generalize two results of \cite{Yap24} from the case of surfaces to arbitrary dimension. For any $n \geq 1$, we will let 
$$\Per_n = \{f^n(x) = x\}$$ 
denote the set of periodic points of period $n$. Given $\alpha \in X(K)$, we will let 
$$f^{-n}(\alpha) = \{ x \in X(\ovl{K}) \mid f^n(x) = \alpha\}.$$
For both $\Per_n$ and $f^{-n}(\alpha)$, we will only consider the underlying set and ignore multiplicity. We let $K(\Per_n)$ denote the field extension obtained by adjoining the coordinates of every point in $\Per_n$. 

\begin{theorem} \label{IntroGaloisTheorem1}
There exists $\lambda > 1$ depending on $f$ such that 
$$[K(\Per_n):K] \geq \lambda^n$$
for all sufficiently large $n$. In fact at least $1 - \lambda^{-n}$ proportion of points of $\Per_n$ have Galois orbit $\geq \lambda^n$ in size for all sufficiently large $n$.  
\end{theorem}

When $X = \bb{P}^1$, this was proven by Baker in \cite{Bak06} for number fields and in \cite{Bak09} for function fields. When $X = A$ is an abelian variety, the number field case can be obtained from results of Masser \cite{Mas84} and the function field case from Looper \cite{Loo24}. The case of number fields when $X$ is a surface was proven by the second author \cite{Yap24} via a quantitative equidistribution theorem for Galois orbits of small points in higher dimensions. Although the quantitative equidistribution theorem applies to all dimensions, some understanding of the geometry of periodic points is required to obtain Theorem \ref{IntroGaloisTheorem1} in higher dimensions, which is the main content of this paper. 
\par 
We also have a similar result for $f^{-n}(\alpha)$ although one has to restrict where $\alpha$ lies in. As in Theorem \ref{IntroGaloisTheorem1}, we will let $K(f^{-n}(a))$ denote the field extension of $K$ obtained by adjoining the coordinates of every point in $f^{-n}(a)$.

\begin{theorem} \label{IntroGaloisTheorem2}
Let $\cal{E}$ denote the maximal totally invariant proper subvariety of $X$ under $f$. There exists $\lambda > 1$, depending only on $f$, such that if $a \not \in \cal{E}$, then
$$[K(f^{-n}(a)):K] \geq \lambda^n$$
for all sufficiently large $n$ depending on $a$. In fact at least $1 - \lambda^{-n}$ proportion of points of $f^{-n}(a)$ have Galois orbits $\geq \lambda^n$ in size for all sufficiently large $n$ depending on $a$.
\end{theorem}

In \cite{Yap24}, a quantitative rate of equidistribution for periodic points was established for polarized endomorphisms $f: X \to X$ defined over a number field $K$ when $X$ is a smooth surface. We similarly generalize this to smooth varieties of arbitrary dimension.

\begin{theorem} \label{IntroQuantitative1}
Let $X$ be a smooth projective variety over a number field $K$ and
$f : X \to X$ a polarized endomorphism of degree $d\geq 2$. There exist constants
$C > 0$ and $\lambda > 1$, depending on $f$, such that if $v$ is an archimedean place of $K$,
then for any smooth function $\vphi : X(\C_v) \to \R$, we have
$$\Big| {1\over |\Per_n|} \sum_{x\in \Per_n} \vphi(x) - \int \vphi d\mu_{f,v}\Big| \leq C \|\vphi\|_{\Cc^3} \lambda^{-n}$$
where $\mu_{f,v}$ is the equilibrium measure of $f$ at the place $v$. 
\end{theorem}

We note that in the case of $X = \bb{P}^1$ and $f$ is defined over a number field $K$, this was already known by Favre--Rivera-Letelier \cite{FRL06} with the exponent $\lambda = d^{1/2 - \delta}$ for any $\delta > 0$, which is likely to be optimal for $C^1$-functions. Recently, Gauthier and Vigny \cite{GV25} were able to extend Favre--Rivera-Letelier's results to the setting of $f$ being defined over $\bb{C}$, with the exponent $\lambda = d^{1/2 - \delta}$ by using Moriwaki heights. For preimages $f^{-n}(a)$ and when $f$ is defined over $\bb{C}$, Drasin and Okuyama \cite{DO07} has obtained near-optimal constants by using Nevalinna theory. Okuyama \cite{Oku12} also has analogous results in the non-archimedean setting. For non-explicit $\lambda$, the general version for endomorphisms $f$ on $\bb{P}^k$ was obtained recently by de-Thelin--Dinh--Kaufmann \cite{dTDK25}. 
\par 
For preimages $f^{-n}(a)$, we can obtain a similar result but as it is weaker than \cite[Theorem 1.1]{DS10b}, we do not include it here. 

\subsection{Acknowledgements} The second author would like to thank Laura DeMarco, Thomas Gauthier and Gabriel Vigny for helpful discussions. The first author has received funding from the National University of
Singapore and MOE of Singapore through the grants A-8002488-00-00 and A-8003576-00-00.

\section{Density of positive closed currents}

In this section, we recall some basic facts on the theory of density of positive closed currents and refer to \cite{DS18} for details.
The theory 
will allow us to bypass Lefschetz's fixed points formula in order to get an upper bound for the number of periodic points or of a negative orbit in a subvariety. 

Let $X$ be a compact K\"ahler manifold of dimension $k$. 
Let $V$ be an irreducible  submanifold of $X$ of  dimension $l$.  We denote by 
$\pi:\ E\to V$ the normal vector bundle  of $V$ in $X$.
For a point $a\in V,$ we denote by $\Tan_a X$ and $\Tan_a V$ the tangent spaces of $X$ and of $V$ at $a$. We also identify 
the fiber $E_a:=\pi^{-1}(a)$ of $E$ over $a$ with the quotient space 
$\Tan_a X / \Tan_a V$ and we identity the zero section of $E$ with $V$. 

Consider the natural compactification $\overline E$ of $E$ that we identify with the projectivisation $\P(E\oplus \C)$ of the vector bundle $E\oplus \C$, where $\C$ is the trivial line bundle over $V$. The above projection $\pi$ is naturally identified to the projection from $\overline E$ to $V$. 
Let $A_\lambda$ be the multiplication by $\lambda$ on the fibers of  $E$ where $\lambda\in\C^*,$
i.e. $A_\lambda(u):=\lambda u,$  $u\in E_a,$ $a\in V.$ This map extends to a holomorphic automorphism of $\overline E$. 

Fix an open subset $V_0$ of $V$ which is naturally  identified with an open subset of the section 0 in $E$. A diffeomorphism $\tau$ from a neighbourhood of $V_0$ in $X$ to a neighbourhood of $V_0$ in $E$ is called 
{\it admissible} if it satisfies the following three conditions:
the restriction of $\tau$ to $V_0$ is the  identity,  the differential  of $\tau$ at each point  $a\in V_0$  is
$\C$-linear and the composition  of 
$$E_a\hookrightarrow \Tan_a(E)\to \Tan_a(X)\to E_a$$    
is the identity. Here, the morphism $\Tan_a(E)\to \Tan_a(X)$ is induced by the differential of $\tau^{-1}$ at $a$ and the other maps are the canonical ones. 
Using exponential maps induced by K\"ahler metrics on $X$, one can construct admissible maps for $V_0=V$. Admissible maps for $V_0=V$ are rarely holomorphic. Here is a main theorem of the density theory for currents.

\begin{theorem}[see {\cite[Section 4]{DS18}}] \label{t:density}
With the above notation, consider an admissible map $\tau$ with $V_0=V$. Let $T$ be a positive closed $(p,p)$-current on $X$ and define $T_\lambda:=(A_\lambda)_*\tau_*(T)$ for $\lambda\in\C^*$. Then the following holds.
\begin{enumerate}
\item[(a)] For any compact set $K\subset E$, when $|\lambda|$ is large enough, the mass of $T_\lambda$ on $K$ is bounded by a constant independent of $\lambda$.
\item[(b)] Let $(\lambda_n)_{n\geq 0}$ be a sequence going to infinity such that $T_{\lambda_n}$ converge to some current $S$ in $E$. Then $S$ can be extended by $0$ to a current on $\overline E$ which is a positive closed $(p,p)$-current that we still denote by $S$.
\item[(c)] The current $S$ may depend on $(\lambda_n)_{n\geq 0}$ but it is independent of the choice of $\tau$. Moreover, its mass is bounded by $c\|T\|$ where $c>0$ is a constant independent of $T$.
\end{enumerate}
\end{theorem}

\begin{remark} \rm
The last property holds for all admissible maps even when $V_0$ is not equal to $V$. More precisely, if $\tau'$ is such a map, we have that $(A_{\lambda_n})_*\tau'_*(T)$ converges to $S$ in $\pi^{-1}(V_0)$. 
\end{remark}

\begin{example}[good local model] \label{e:local} \rm
When $V_0$ is small enough, there are local holomorphic coordinates on a small neighbourhood $U$ of $V_0$ in $X$ so that over $V_0$ we identify naturally $E$ with $V_0\times \C^{k-l}$ and $U$ with an open neighbourhood of $V_0\times\{0\}$  in $V_0\times \C^{k-l}$. 
In this setting, the identity is a holomorphic admissible map. With a suitable choice of local coordinates $z=(z',z'')\in \C^l\times \C^{k-l}$, we have that $V_0$ is given by the equation $z''=0$ and 
the map $A_\lambda$ is given by $z\mapsto (z',\lambda z'')$.
\end{example}

Recall that if $Z$ is a subvariety of dimension $q$ in a compact K\"ahler manifold $(Y,\omega_Y)$, then the degree of $Z$ is defined by 
$$\deg(Z):=\int_Z \omega_Y^q.$$
By Wirtinger's equality, $\deg Z$ is equal to $q!$ times the $2q$-dimensional volume of $Z$. The following holds for subvarieties of any dimensions. Note that in the algebraic setting, we can also obtain this result using a suitable moving lemma.

\begin{corollary} \label{c:intersection}
Let $Y$ be a compact K\"ahler manifold endowed with a K\"ahler form $\omega_Y$. Let $Z_1$ and $Z_2$ be two subvarieties of $X$. Then the number of isolated points in $Z_1\cap Z_2$ is bounded by $c\deg(Z_1)\deg(Z_2)$ for some constant $c>0$ independent of $Z_1$ and $Z_2$.
\end{corollary}
\proof
Consider $X:=Y\times Y$, $V$ the diagonal of $X$ and $Z:=Z_1\times Z_2$. Let $\pi_1$ and $\pi_2$ be the natural projections from $X$ onto its factor $Y$. We will use the K\"ahler form on $X$ defined by $\omega:=\pi_1^*(\omega_Y)+\pi_2^*(\omega_Y)$ for which the degree of $Z$ is equal to ${q_1 + q_2 \choose q_1} \deg(Z_1)\times\deg(Z_2)$, where $q_i$ is the dimension of $Z_i$.
Let $T=[Z]$ be the current of integration on $Z$ and we will use the notation from Theorem \ref{t:density}. 
Let $a$ be an isolated point in $Z_1\cap Z_2$ and $\widetilde a:=(a,a)$ a point of $V$. By Theorem \ref{t:density}(c), the mass of $S$ is bounded by a constant times 
$\deg(Z_1)\times\deg(Z_2)$. Therefore, in order to get the result, it is enough to check that the mass of $S$ on $\pi^{-1}(\widetilde a)$ is bounded from below by a strictly positive constant independent of $Z_1$ and $Z_2$.

We use now the map $\tau$ defined in Example \ref{e:local} for $V_0$ and for a local coordinate system centered at the point $\widetilde a$. In this setting, we have $S=\lim_{n\to\infty} (A_{\lambda_n})_*(T)$. Since $\widetilde a$ is isolated in the intersection $Z\cap V$, we see that in $\pi^{-1}(V_0)$ the current $S$ is supported by $\pi^{-1}(\widetilde a)$. 
Now, since $(A_{\lambda_n})_*(T)$ is given by an analytic set, its Lelong number at $\widetilde a$ is at least equal to 1. By the upper semi-continuity theorem of Lelong numbers in the current variable \cite{Dem93}, $S$ has a Lelong number at $\widetilde a$ at least equal to 1. It follows that its mass is bounded from below by a constant. This ends the proof of the corollary.
\endproof

\section{Upper bound for periodic points and for negative orbits}

In this section, we prove the following uniform bounds for periodic points and for negative orbits on analytic sets which generalizes \cite[Proposition 9.3]{Yap24}. The results can be extended to a more general setting of meromorphic maps or correspondences. 

\begin{proposition} \label{IntersectionBound1}
Let $f:\P^k\to\P^k$ be a holomorphic endomorphism of algebraic degree $d\geq 2$ of $\P^k$. Let $\Per_n$ denote the set of periodic points of period $n$ of $f$.
There is a constant $c>0$ which only depends on $k$ such that if $Z$ is an analytic subset of degree $e$ and of dimension $q$ of $\P^k$  and if $a$ is a point in $\P^k$ then
$$\left| \Per_n\cap Z  \right|\leq c e d^{qn} \qquad \text{and} \qquad \left| f^{-n}(a)\cap Z \right|\leq ced^{qn}$$
for every $n\geq 1$.
\end{proposition}

\proof
Define $X:=\P^k\times \P^k$. Let $V$ be the diagonal of $X$ and $V':=\P^k\times \{a\}$. Consider 
$$Z_n:=\{(x,f^n(x)): \, x\in Z\}.$$
We have $\left| \Per_n\cap Z \right| = \left| Z_n\cap V \right|$ and $\left| f^{-n}(a)\cap Z \right| = \left| Z_n\cap V' \right|$. Observe that all these intersections are finite sets. 
By Corollary \ref{c:intersection}, it is enough to show that $\deg Z_n \leq c e d^n$ for some constant $c>0$ which only depends on $k$.

Let $\omega_\FS$ be the standard Fubiny-Study form on $\P^k$ which is in the cohomology class of a projective hyperplane. We consider the K\"ahler form $\omega:=\pi_1^*(\omega_\FS)+ \pi_2^*(\omega_\FS)$ where $\pi_1$ and $\pi_2$ denote the standard projections from $X$ to its factors. 
Then we have 
\begin{eqnarray*}
\deg Z_n &=& \int_{Z_n}\omega^q = \sum_{j=0}^q {q\choose j} \int_{Z_n} \pi_1^*(\omega_\FS)^{q-j} \wedge \pi_2^*(\omega_\FS)^j \\
&=& \sum_{j=0}^q {q\choose j} \int_Z \omega_\FS^{q-j} \wedge (f^n)^*(\omega_\FS)^j.
\end{eqnarray*}
The last integral only depends on $Z$ and on the cohomology class of the form $(f^n)^*(\omega_\FS)^j$. It is known that this form is cohomologous to $d^{nj}$ times the cohomology class of $\omega_\FS^j$. Therefore, we get
$$\deg Z_n = \sum_{j=0}^q {q\choose j} \int_Z \omega_\FS^{q-j} \wedge d^{nj} \omega_\FS^j = \sum_{j=0}^q {q\choose j} d^{nj} \deg Z = e (d^n+1)^q.$$
The proposition follows. 
\endproof

\begin{remark} \rm
We can consider an analytic set $A$ instead of the point $a$. We get a similar result for the number of isolated points in $f^{-n}(A)\cap Z$.  
\end{remark}

\section{Proofs of main theorems}
We mainly follow the arguments in \cite{Yap24}. We first recall \cite[Theorem 8.1]{Yap24}. 

\begin{theorem} \cite[Theorem 8.1]{Yap24} \label{QuantEquidistribution1}
Let $X$ be a smooth variety defined over a number field $K$ and let $f: X \to X$ be a polarized endomorphism of degree $d \geq 2$, i.e. there exists an ample line bundle $L$ such that $f^*L \simeq L^d$. There exists constants $c_1 = c_1(K,f), c_2 = c_2(f), c_3 = c_3(K,f)$ such that for any $\eps > 0$ and archimedean place $v$ of $K$, if $\vphi: X(\ovl{K}_v) \to \bb{R}$ is a smooth function, then 

\begin{enumerate}
\item there exists a hypersurface $H(\vphi,\eps)$ defined over $K$ with degree at most $c_1 \eps^{-c_2}$ with respect to $L$,

\item for all $x \in X(\ovl{K})$ with Galois orbit $F_x$, if $x \not \in H(f,\eps)$ then 
$$\left|\frac{1}{|F_x|} \sum_{y \in F_x} \vphi(y) - \int \vphi(y) d \mu_{f,v}(y) \right| \leq \|\vphi\|_{\cal{C}^3} \left(\frac{\h_f(x)}{\eps} + c_3 \eps \right).$$
\end{enumerate}
\end{theorem}

Here, the quantity $\h_f$ is the canonical height with respect to $f$, which was first defined in this setting by Call--Silverman \cite{CS93}. The canonical height satisfies $\h_f(f(x)) = d \h_f(x)$ and that $\h_f(x) = 0$ if and only if $x$ is a preperiodic point. First, by \cite{Fak03}, after passing to a suitable power of $L$ we may use $L$ to embed $X$ into $\bb{P}^N$ with an endomorphism $F: \bb{P}^N \to \bb{P}^N$ that extends $f$ on $X$. 

\begin{proof}[Proof of Theorem \ref{IntroQuantitative1}]
Let $\vphi$ be our given test function and let $\dim X = k$. By Theorem \ref{QuantEquidistribution1}, there exists a hypersurface $H(\vphi,\eps)$ such that if $x$ is a periodic point, so that $\h_f(x) = 0$ and $x \not \in H(\vphi,\eps)$, then 
$$\left|\frac{1}{|F_x|} \sum_{y \in F_x} \vphi(y) - \int \vphi(y) d \mu_{f,v}(y) \right| \leq \|\vphi\|_{\cal{C}^3} (c_3 \eps).$$
Furthermore, we may bound the degree of $H(\vphi,\eps)$ by $c_1 \eps^{-c_2}$. By \cite{DZ23}, there is a constant $\alpha > 0$ depending on $f$ such that $|\Per_n| \geq \alpha d^{kn}$. Applying Proposition \ref{IntersectionBound1}, we obtain $|\Per_n \cap H(\vphi,\eps)| \leq c (c_1 \eps^{-c_2}) d^{(k-1)n}$ as $\dim H(\vphi,\eps) = k-1$. Now pick a $\lambda > 1$ such that $d (\lambda)^{-c_2} > \lambda.$ Then taking $\eps = \lambda^{-n}$, we obtain that
$$|\Per_n \cap H(\vphi,\lambda^n)| \leq c c_1 d^{kn} (\lambda)^{-n}$$
and in particular, at least $1 - \lambda^{-n}$ proportion of points of $\Per_n$ do not lie on $H(\vphi, \lambda^{-n})$. If $G_x$ denotes the subset of $\Per_n$ that does not lie on $H(\vphi,\lambda^{-n})$, then $G_x$ is $\Gal(\ovl{K}/K)$-stable and we have  
$$\left|\frac{1}{|G_x|} \sum_{y \in G_x} \vphi(y) - \int \vphi(y) d \mu_{f,v} \right| \leq \|\vphi\|_{\cal{C}^3} c_3 (\lambda)^{-n}.$$
For $y \in \Per_n \setminus G_x$, we may bound $|\vphi(y)|$ by $\|\vphi\|_{\cal{C}^3}$ and since $\frac{|G_x|}{|\Per_n|} \geq 1 - \lambda^{-n}$ for sufficiently large $n$, we obtain 
$$\left|\frac{1}{|\Per_n|} \sum_{y \in \Per_n} \vphi(y) - \int \vphi(y) d \mu_{f,v} \right| \leq C \|\vphi\|_{\cal{C}^3} \lambda^{-n}$$
for some $C > 0$ depending only on $f$ as desired.
\end{proof}

To prove Theorems \ref{IntroGaloisTheorem1} and \ref{IntroGaloisTheorem2}, we first recall the following proposition. We first fix an open set $U$ of $X$ that is isomorphic to $\bb{D}_2$, the open disc of radius $2$, such that $\mu_{f,v}(\bb{D}_1) = \delta > 0$. If $R$ is a cube in $\bb{C}^k \simeq \bb{R}^{2k}$, we let $nR$ denote the cube with the same center as $R$ but with sides having $n$ times the length. 

\begin{proposition} \cite[Proposition 10.2]{Yap24}\label{Measure1}
Fix a positive integer $n$. There exist constants $c, c' > 0$ depending on $f$ and $n$ such that for any positive integer $D \geq 2$,  there are $D$ cubes $\{R_i\}$, each of length $c'D^{-2 \kappa}$ where $\kappa > 0$ depends on $f$, such that $\mu_{f,v}(R_i) \geq c D^{-4k \kappa}$ and the cubes $\{nR_i\}$ are all pairwise disjoint. 
\end{proposition}

\begin{proof}
This is Proposition 10.2 of \cite{Yap24} except there are some calculation mistakes and typos there and we will rewrite the proof here. We choose $C, \kappa > 0$ such that for any cube $R$ of length $\eps$ in our chart, we have 
$$\mu_{f,v}(R) \leq C \eps^{1/\kappa}.$$
Now partition $\bb{D}_2$ into cubes of length $c' D^{- 2 \kappa}$. Then we have $(2 c' D^{- 2 \kappa})^{-2k}$ many such cubes and hence among the cubes covering $\bb{D}_1$, there must be one with measure at least $\delta (2c' D^{-2\kappa})^{2k}$. 
\par 
If $R$ is this cube, then $nR$ is a cube of length $n c' D^{-2 \kappa}$ and so $\mu_{\vphi,v}(nR) \leq n (c')^{1/\kappa} C D^{-2}$. We choose $c'$ so that $n(c')^{1/\kappa} C \leq \delta$. Then the remaining cubes which do not intersect $nR$ will have measure at least $\delta (1- D^{-2})$ left. This allows us to choose a cube $R_2$ with measure at least $\delta (1-D^{-2})(2c' D^{-2\kappa})^{2k}$. The remaining cubes will then have measure at least $\delta (1 - 2D^{-2})$. Repeating this, we obtain $D$ cubes $R_1,\ldots,R_D$ with $nR_i$ all disjoint, such that 
$$\mu_{\vphi,v}(R_i) \geq \delta(1-D^{-1})(2c')^{2k}  D^{-4 \kappa k}.$$
Since $D \geq 2$, this gives us $c = \frac{1}{2} \delta (2c')^{2k}$ as desired.
\end{proof}

\begin{proof}[Proof of Theorem \ref{IntroGaloisTheorem1}]
We let $\lambda > 1$ be a constant that we will choose later and we set $D = \lambda^n$ and $n = 2$. Then we may find $D$ many cubes $R_i$ with $\mu(R_i) \geq c \lambda^{- 4k \kappa n }$ such that $2R_i$ are all pairwise disjoint. The length of each $_Ri$ is $O_{\vphi}(D^{-2 \kappa})$. We let $\vphi_i$ be a non-negative function that is $1$ on $R_i$ and $0$ on $2R_i$. Such a function $\vphi_i$ may be chosen so that 
$$\|\vphi_i\|_{\cal{C}^3} \leq O_{f}(\lambda^{6k \kappa n})$$
and applying Theorem \ref{QuantEquidistribution1}, we obtain a hypersurface $H(\vphi_i,\eps)$ such that if $x \in \Per_n$ with $x \not \in H(\vphi_i,\eps)$, we have 
$$\left|\frac{1}{|F_x|} \sum_{y \in F_x} \vphi_i(y) - \int \vphi_i(y) d \mu_{f,v}(y) \right| \leq O_{f}(\lambda^{6k \kappa n}) (c_3 \eps).$$
As $\vphi_i = 1$ on $R_i$ and $\mu(R_i) \geq c \lambda^{-4k \kappa n}$, it follows that 
$$\int \vphi_i(y) d \mu_{f,v}(y) \geq c \lambda^{-4k \kappa n}$$
and so taking $\eps = c' \lambda^{-10 k \kappa n}$ for some $c'$ sufficiently small depending on $f$, we must have $\sum_{y \in F_x} \vphi_i(y) > 0$ and so $F_x \cap 2R_i \not = \emptyset$. Since $\{2R_i\}$ are pairwise disjoint, if $x \in \Per_n$ and $x \not \in H(\vphi_i,\eps)$ for any $i$, it follows that $|F_x| \geq D$ and thus  $|F_x| \geq \lambda^n$. 
\par 
We may now bound the degree of each $H(\vphi_i,\eps)$ by $c_1 O_{f}(\lambda^{10 k \kappa n})^{c_2}$ and so taking the union over all $i$'s, we have a total degree of $ O_{f}(\lambda^{10 c_2 \kappa n})$. We now take $\lambda > 1$ so that 
$$O_{\vphi}(\lambda^{10 c_2 \kappa n}) < d^{n} \lambda^{-n}$$
and so by Proposition \ref{IntersectionBound1}, we have 
$$\left|\Per_n \cap \left(\bigcup_{i=1}^{D} H(\vphi_i,\eps) \right) \right| \leq cd^{kn} \lambda^{-n}$$
for some constant $c> 0$ depending only on $f$. Since $|\Per_n| \geq \alpha d^{kn}$, after reducing $\lambda$ we get that at least $1 - \lambda^{-n}$ proportion of points of $\Per_n$ do not lie on any of $H(\vphi_i,\eps)$ and for each of these points $x$, we must have $|F_x| \geq \lambda^n$ as desired.
\end{proof}

For the corresponding result of negative orbits, we have to show that the size of $|f^{-n}(\alpha)|$ grows larger than $d^{(k-1)n}$. Recall from \cite{Din09} that if $\kappa_n(x)$ denotes the multiplicity of $f^n$ at $x$, then $\{\kappa_n\}_{n \geq 1}$ forms an analytic multiplicative cocycle. We can extend this to negative integers by setting 
$$\kappa_{-n}(x) = \max_{y \in f^{-n}(x)} \kappa_n(y).$$
By \cite[Theorem 1.2]{Din09}, the function $\kappa_{-n}^{1/n}$ converges pointwise to a function $\kappa_{-}$. Furthermore, for every $\delta > 1$, the level set $\{\kappa_{-} \geq \delta\}$ is a proper analytic subset of $X$ which is invariant under $f$. 
\par 
Using the discussion after \cite[Theorem 3.1]{Din09}, we know that there is a proper analytic subset $\cal{E}$ of $X$ that is totally invariant by $f$, i.e. $f^{-1}(\cal{E}) = f(\cal{E}) = \cal{E}$, and that $\cal{E}$ is maximal. We are now ready to prove Theorem \ref{IntroGaloisTheorem2}.

\begin{proof}[Proof of Theorem \ref{IntroGaloisTheorem2}]
Given a proper analytic subset $Y$ of $X$ and $a \in X \setminus \cal{E}$, by maximality of $\cal{E}$ there exists $m$ such that there exists $b \in f^{-m}(a)$ with $b \not \in Y$. Fix $\lambda_0 > 1$ which is $< \lambda$ for Theorem \ref{IntroQuantitative1} and a $\delta > 0$ so that $\lambda - \delta > 1$ still. We set $\lambda_1 = \lambda - \delta$. We apply this to 
$$Y = E_{\lambda_1} = \{\kappa_{-} \geq \lambda_1 \}$$
Then by definition, the multiplicity of $f^n$ at $c$ for which $f^n(c) = b$ is at most $\lambda_1^n$ for $n$ sufficiently large and hence the number of distinct elements of $f^{-n}(b) \subseteq f^{-n-m}(a)$ is at least $d^{kn} \lambda_1^{-n}$ for $n$ sufficiently large. The rest of the argument is then similar to the proof of Theorem \ref{IntroGaloisTheorem1}. 
\end{proof}





\printbibliography

\end{document}